\documentclass[a4paper,10pt]{article}

\usepackage{amsthm}
\usepackage{amsmath}
\usepackage{amsxtra}
\usepackage{amssymb}
\usepackage{comment}
\usepackage{enumitem}
\usepackage{undertilde}
\usepackage{mathrsfs}
\usepackage{thmtools}
\usepackage{hyperref}
\usepackage{bbm}
\newcommand{\RR}{\mathbb R}

\newcommand{\PP}{\mathbb P}
\newcommand{\BB}{\mathbb B}
\newcommand{\sub}{\subseteq}
\newcommand{\cross}{\times}

\newcommand{\inter}{\cap}
\renewcommand{\int}{\inter}

\newcommand{\om}{\omega}

\newcommand{\OR}{\mathrm{OR}}

\newcommand{\cut}{\backslash}

\newcommand{\Tt}{\mathcal{T}}

\newcommand{\Uu}{\mathcal{U}}

\newcommand{\rg}{\mathrm{rg}}

\newcommand{\crit}{\mathrm{cr}}

\newcommand{\rest}{\!\upharpoonright\!}
\newcommand{\com}{\circ}

\newcommand{\lh}{\mathrm{lh}}
\newcommand{\Ult}{\mathrm{Ult}}

\newcommand{\sats}{\models}

\newcommand{\HC}{\mathrm{HC}}
\newcommand{\ZFC}{\mathsf{ZFC}}

\newcommand{\Coll}{\mathrm{Col}}

\newcommand{\id}{\mathrm{id}}

\newcommand{\conc}{\ \widehat{\ }\ }

\DeclareMathOperator{\card}{card}
\DeclareMathOperator{\cof}{cof}

\DeclareMathOperator{\rank}{rank}

\newcommand{\OD}{\mathrm{OD}}

\newcommand{\Yy}{\mathcal{Y}}

\newcommand{\Xx}{\mathcal{X}}

\declaretheorem[name=Definition,style=definition,qed=$\dashv$,numberwithin=section]{definition}
\declaretheorem[name=Definition,style=definition,numbered=no,qed=$\dashv$]{definition*}
\declaretheorem[name=Definition,style=definition,qed=$\dashv$,sibling=definition]{dfn}
\declaretheorem[name=Definition,style=definition,numbered=no,qed=$\dashv$]{dfn*}

\declaretheorem[name=Theorem,style=plain,sibling=dfn]{tm}
\declaretheorem[name=Theorem,style=plain,sibling=dfn]{tm*}

\declaretheorem[name=Theorem,style=plain,numbered=no]{theorem*}

\declaretheorem[name=Lemma,style=plain,sibling=dfn]{lem}

\declaretheorem[name=Corollary,style=plain,numbered=no]{corollary*}
\declaretheorem[name=Corollary,style=plain,sibling=dfn]{cor}
\declaretheorem[name=Corollary,style=plain,numbered=no]{cor*}
\declaretheorem[name=Remark,style=definition,sibling=dfn]{rem}

\declaretheoremstyle[headfont=\scshape]{claimstyle}
\declaretheorem[name=Claim,style=claimstyle]{clm}

\declaretheorem[name=Subclaim,style=claimstyle]{sclm}

\newcommand{\lex}{\mathrm{lex}}

\newcommand{\Hom}{\mathrm{Hom}}
\newcommand{\strength}{\mathrm{strength}}
\newcommand{\WO}{\mathrm{WO}}
\newcommand{\leftmost}{\mathrm{left}}
\author{Farmer Schlutzenberg and John R.~Steel}
\title{Stationary tower free homogeneously Suslin scales}
\begin{document}

\maketitle
\begin{abstract}
Let $\lambda$ be a limit of Woodin cardinals.
It was shown by the second author that 
the pointclass of ${<\lambda}$-homogeneously Suslin sets has the scale property.
We give a new proof of this fact,
which avoids the use of stationary tower forcing.
\end{abstract}

\section{Introduction}

We work in $\ZFC$. A key tool used in the proof of Woodin's derived model theorem
(see \cite{steel_dmt})
is Woodin's stationary tower forcing (see \cite{stattower}).
In \cite{stfdmt},
the second author gave
a new proof of the derived model theorem (the older version thereof, that is),
with arguments involving iteration trees replacing
the central uses of the stationary tower in the proof. However, the
proof still made implicit use of 
the stationary tower,
because it appealed to the following theorem of 
the second author
(see \cite{steel_dmt} and \cite{stfdmt} for this proof and other definitions and background):

\begin{tm}[Steel]\label{thm:Hom_scale} Let $\lambda$ be a limit of Woodin cardinals.
Then every $\Hom_{<\lambda}$ set of reals has a $\Hom_{<\lambda}$ scale.
\end{tm}

The standard proof of Theorem \ref{thm:Hom_scale} (see \cite[Theorem 
4.3]{steel_dmt}) also uses the stationary tower.
This is the only remaining appeal to the stationary tower in
the proof given in \cite{stfdmt}.
Given the title of \cite{stfdmt}, it therefore seems a worthy aim
to find a proof of Theorem \ref{thm:Hom_scale}
which also avoids 
appeal to the stationary tower.
This is the burden of this note.
We remark that there is still no known proof of the
Woodin's improved derived model 
theorem which avoids the stationary tower.

Recall the following theorem, which follows from various results of Martin, Steel and Woodin:

\begin{tm}[Martin, Steel, Woodin]\label{thm:MSW}
 Let $\lambda$ be a limit of Woodin cardinals and $A\sub\RR$.
 Then the following are equivalent:
 \begin{enumerate}[label=--]
  \item $A\in\Hom_{<\lambda}$,
  \item $A$ is $\lambda$-universally Baire,
  \item $A$ is ${<\lambda}$-weakly homogeneously Suslin.
 \end{enumerate}
\end{tm}

One component of the standard proof of this result is the following theorem of Woodin:

\begin{tm}[Woodin]\label{thm:Woodin}
 Let $\delta$ be Woodin. Let $A\sub\RR$ be $(\delta+1)$-universally Baire.
 Then $A$ is ${<\delta}$-weakly homogeneously Suslin.
\end{tm}

Woodin's original proof of Theorem \ref{thm:Woodin} uses the stationary tower. 
However, Steel gave a stationary tower free proof; see \cite[Theorem 
1.2]{stfdmt}.
So we are free to use Theorem \ref{thm:Woodin}, and hence Theorem \ref{thm:MSW}.

\begin{rem}\label{rem:bug}
In the second author's proof just mentioned, he considers
a countable transitive $M$ and elementary
$\pi:M\to V_\eta$, with $\rg(\pi)$ including the given Woodin cardinal 
$\delta$ and a pair of $(\delta+1)$-absolutely 
complementing trees.
By Martin-Steel \cite{itertrees}, the structure $(M,\bar{\delta})$ is  $(\om+1)$-iterable 
for extender-nice trees,
where $\pi(\bar{\delta})=\delta$
(see \ref{dfn:nice}),
and this is
 enough to use Neeman genericity iterations on $(M,\bar{\delta})$,
 which come up in the proof.
For the stationary tower free proof of \ref{thm:Hom_scale} we will want an $M$ similar to this,
but also such that $M$ is iterable for certain uncountable length nice trees
(those based on an a certain interval
$(\bar{\theta},\bar{\delta})$).
In \cite[Remark 3.2]{steel_dmt} there is a sketch of another alleged proof of Theorem \ref{thm:Woodin}, also stationary tower free,
in which this sort of iterability is claimed.
However, the first author noticed a gap in that argument, for which a repair is not obvious.
Let us first explain what the issue there is, so as to make it clear why we don't just find our $M$ 
as in that remark.

In the notation of \cite[Remark 3.2]{steel_dmt}, in its second paragraph, the idea seems to be as follows.
We have trees $T$ and $U$,\footnote{Actually only $T$ is mentioned in the second paragraph, but in the end, one wants a countable iterable structure with a pair of absolutely complementing trees,
in order to argue as in the first paragraph of \cite[Remark 3.2]{steel_dmt}.}
chosen as substructures of a pair of $(\delta+1)$-absolutely complementing trees,
such that $T,U$ are on $\om\cross\delta$ and project to complements in generic extensions by the extender algebra $\BB_\delta$ at $\delta$.
In  the last sentence of \cite[Remark 3.2]{steel_dmt} there is the suggestion to ``in the extender algebra,
only use identities induced by $(T,U)$-strong extenders'', in order to ensure
the relevant iterability.
But this means that we are now considering some other extender algebra
$\BB_{T,U,\delta}$, which moreover depends on $T$ and $U$. Thus, it is possible that $\BB_\delta\neq\BB_{T,U,\delta}$,
so it seems possible that $T,U$ do not project to complements in extensions by $\BB_{T,U,\delta}$,
which breaks the argument. (A further issue is that after taking $\delta$ least such that
$\delta$ is Woodin in $L(V_\delta,T\rest\delta,U\rest\delta)$, it is not clear that $T\rest\delta,U\rest\delta$
project to complements in extensions by the extender algebra of that model.)
\end{rem}

The key lemma toward the original proof of Theorem \ref{thm:Hom_scale} is the following:
\begin{tm}[Steel]\label{thm:steel_UB_scale}
 Let $\delta$ be  Woodin  and $T$ be a $(\delta+1)$-homogeneous tree.
 Then there is a $\delta$-universally Baire scale on $\RR\cut p[T]$.
\end{tm}

Using Theorems \ref{thm:steel_UB_scale}, \ref{thm:Woodin} and \ref{thm:MSW},  Theorem \ref{thm:Hom_scale} is established as follows:
\begin{proof}[Proof of Theorem \ref{thm:Hom_scale}]
Let $A\in\Hom_{<\lambda}$. By 
Steel-Woodin \cite[Theorem 2.7]{steel_dmt},
\[ \Hom_{<\lambda}=\Hom_{\gamma_0}\text{ for some }\gamma_0<\lambda.\]
Let $\gamma_0<\delta_0<\delta_1<\delta_2<\delta_3$ with the $\delta_i$'s Woodin. Then by \ref{thm:MSW}, $\RR\cut A$ is 
$(\delta_3+1)$-universally Baire. So by \ref{thm:Woodin},
$\RR\cut A$ is ${<\delta_3}$-weakly homogeneously Suslin,
and hence $\RR\cut A=p[A']$ for some ${<\delta_3}$-homogeneously Suslin set $A'\sub\RR^2$. So by \ref{thm:steel_UB_scale}, there is a scale $B$ on $A$ which is 
$\delta_2$-universally Baire. By \ref{thm:Woodin},
$B$ is ${<\delta_1}$-weakly homogeneously Suslin, so by Martin and Steel \cite{projdet}, $B$ is $<\delta_0$-homogeneously Suslin,
so $B\in\Hom_{\gamma_0}=\Hom_{<\lambda}$.\end{proof}

Now we do not see how to prove Theorem \ref{thm:steel_UB_scale} without appealing to the stationary tower. But
note that the proof above shows:
`\begin{cor}[Steel]\label{cor:local}
Let $\delta_0<\delta_1<\delta_2<\delta_3$ be Woodin cardinals and let $A\sub\RR$ be  
$(\delta_3+1)$-universally Baire. Then there is a scale on $A$ which is ${<\delta_0}$-homogeneously Suslin.
\end{cor}

Clearly it suffices to give a stationary tower free proof of this corollary; this is what we will do.

\begin{rem}
 The first author noticed the implicit dependence on the proof of \cite{stfdmt}
 on the stationary tower mentioned above, and also the issue mentioned in Remark \ref{rem:bug}, and a fix to that issue (the fix is as done in the proof of Lemma \ref{lem:capturing_mouse}) and mentioned these things to the second author during the M\"unster inner model theory conference in 2017. The two authors then found the stationary tower free scale construction presented here, during that conference.
\end{rem}

\section{Stationary tower free proof of Corollary \ref{cor:local}}

\begin{lem}\label{lem:DC_in_L(V_delta,A)}
 Assume ZF, $V=L(V_\delta,A)$ for some class of ordinals $A$, $V_\delta\models$ ZFC  and $\delta$ is Woodin. Then DC  holds and  ${<\delta}$-choice holds.
\end{lem}
\begin{proof}
Consider DC. Let $X$ be a set and let $R$ be a non-empty  set of finite tuples such that for all $\sigma\in R$
there is $x\in X$ such that $\sigma\conc\left<x\right>\in R$. We need to show that an infinite branch through $R$ exists.
 Because $V=L(V_\delta,A)$, a standard
 calculation shows that we may assume $X=V_\delta$. Since $\delta$ is Woodin,
 we can fix $\kappa$ which is $({<\delta},R)$-reflecting. Let $\bar{R}=R\cap V_\kappa$.
 It is then easy to see that $\bar{R}\neq\emptyset$,
 and that for every $\sigma\in\bar{R}$,
 there is $x\in V_\kappa$ such that $\sigma\conc\left<x\right>\in\bar{R}$.
 But $\bar{R}\in V_\delta\models$ ZFC,
 and therefore there is an infinite branch through $\bar{R}$, which suffices.
 
For ${<\delta}$-choice,
let $\gamma<\delta$ and let $F:\gamma\to V$ be some function with $F(\alpha)\neq\emptyset$ for all $\alpha<\gamma$. Then for each $\alpha<\gamma$,
there is some least $\xi_\alpha<\delta$ such that there is some $x\in V_{\xi_\alpha}$
such that $F(\alpha)$ has an element which is $\OD_x$. Let $\xi=\sup_{\alpha<\gamma}\xi_\alpha$;
so $\xi<\delta$. But then by choice in $V_\delta$, we can find a function $\alpha\mapsto x_\alpha$, with domain $\gamma$, and such that
for each $\alpha<\gamma$, $x_\alpha\in V_\xi$
and there is an element of $F(\alpha)$
which is $\OD_{x_\alpha}$.
This easily suffices.
\end{proof}
\begin{dfn}\label{dfn:nice}
Recall that a (short) extender $E$ is \emph{nice} iff $\strength(E)=\nu_E$ (where $\nu_E$ denotes the strict supremum of the generators of $E$) is an inaccessible 
cardinal.

We say that an iteration tree $\Tt$ is \emph{extender-nice} iff for every $\alpha+1<\lh(\Tt)$,
$M^\Tt_\alpha\sats$``$E^\Tt_\alpha$ is nice''. Recall that $\Tt$ is \emph{nice} if it is
normal and extender-nice.

A (partial) $\kappa$-iteration strategy is \emph{decent} if it applies  to all extender-nice trees 
of length $<\kappa$.
\end{dfn}

\begin{lem}\label{lem:capturing_mouse}
 Let $\delta_2<\delta_3$ be Woodins and let $U,W$ be $(\delta_3+1)$-absolutely complementing 
trees on $\om\cross\lambda$ for some $\lambda\in\OR$. Then there is a countable 
coarse premouse 
$(M,\delta^M)$ and $S^M,T^M\in M$ and $\Sigma$ such that:
\begin{enumerate}
 \item $(M,\delta^M)$ satisfies the conditions
 in the first bullet point of \cite[Definition 3.1]{iter_for_stacks}\footnote{For the reader's convenience, that is, 
then $\delta^M\leq\OR^M=\rank(M)$, $\delta^M$ and $\OR^M$ are limit ordinals, $\card^M(V_\alpha^M)<\delta^M$ for all $\alpha<\delta^M$, $\cof^M(\delta^M)$ is not measurable in $M$, $M$ satisfies $\Sigma_0$-comprehension and is rudimentarily closed,
and $M$ satisfies $\lambda$-choice for all $\lambda<\delta_\eta$.}, and satisfies AC,
 \item\label{item:Sigma_is_strategy} $\Sigma$ is a decent $\delta_2$-iteration strategy for $M$, 
with strong hull condensation,
 \item $\Sigma\rest\HC$ is $<\delta_2$-homogeneously Suslin,
 \item\label{item:M_has_trees} $M\sats$``$\delta^M$ is Woodin and $S^M,T^M$ are 
$(\delta^M+1)$-absolutely projecting'',
 \item\label{item:capturing_trees} For any $\Sigma$-iterate $P$ of $M$, via a tree $\Tt$ of length 
$\kappa+1<\delta_2$, and 
for any $V$-generic $G\sub\Coll(\om,\kappa)$, we have
 \[ p[S^P]^{V[G]}\sub p[U]^{V[G]}\text{ and }p[T^P]^{V[G]}\sub p[W]^{V[G]}.\]
 Hence, $V[G]\sats$``if $x\in\RR$ is $P$-generic for $\Coll(\om,\delta^P)$ then $x\in p[S^P]$ iff 
$x\in p[U]$
 and $x\in p[T^P]$ iff $x\in p[W]$''.
 \end{enumerate}
\end{lem}

\begin{proof}[Proof  (stationary tower free)]
Let $U,W$ be a pair of $(\delta_3+1)$-absolutely complementing trees.
By \ref{thm:Woodin} we may fix $(\delta_2+1)$-complete weak homogeneity systems
$\left<\mu_{st}\right>_{s,t\in{^{<\om}\om}}$ and $\left<\mu'_{st}\right>_{s,t\in{^{<\om}\om}}$
which project to $p[U],p[W]=\RR\cut p[U]$ respectively.
Let $\xi\in\OR$ with $\left<\mu_{st}\right>_{x,t},\left<\mu'_{st}\right>_{s,t}\in V_\xi$. Let $T,S$ 
be the Martin-Solovay trees respectively for $\RR\cut p[U]$  and $\RR\cut p[W]$,
up to some strong limit cardinal $\gamma$ with $\cof(\gamma)>\xi$.
 So in particular, $p[S]=p[U]$ and $p[T]=p[W]$.

\begin{clm}\label{clm:S,T_fixed}
Let $\eta<\delta_2$ and $\PP\in V_{\eta+1}$ and $G$ be $(V,\PP)$-generic (possibly $\PP=G=\emptyset$). 
Let $\Tt\in V[G]$ be any iteration tree on $V$ based on the interval $(\eta,\delta_2)$,  of successor length 
$\leq\delta_2+1$  (in particular, $\crit(E^\Tt_\alpha)>\eta$ for all $\alpha+1<\lh(\Tt)$). Then $i^\Tt(S,T)=(S,T)$.
\end{clm}
\begin{proof}
The proof is essentially a direct transcription of the proofs of Lemma 4.5 and Corollary 4.6 of \cite{steel_dmt}, which we leave to the reader. (Since $\Tt$ is above $\eta+1$, there is an equivalent tree $\Tt^+$ on $V[G]$. Moreover, there are weak homogeneity systems $\vec{\mu}^+,\vec{\mu'}^+$ of $V[G]$ which are equivalent to $\vec{\mu}$ and $\vec{\mu'}$, and letting $T^+,S^+$ be the resulting Martin-Solovay trees as computed in $V[G]$, we have $T^+=T$ and $S^+=S$.
 Now the stationary tower embedding  used in \cite{steel_dmt} is replaced by $i^{\Tt^+}$ here, and things work analogously to in \cite{steel_dmt}
because $\Tt^+$ is based on $V_{\delta_2}^{V[G]}$ and has length $\leq\delta_2+1$,
and by  the $(\delta_2+1)$-completeness of the weak homogeneity systems.)
\end{proof}

 Let $W$ be a wellorder of $V_{\delta_2}$ in ordertype $\delta_2$,
 such that for $V_\gamma$ is an initial segment under $W$ for every ordinal $\gamma<\delta_2$.
Given $\eta$ with $2<\eta<\delta_2$, let $\delta_\eta$ be the least $\delta>\eta$ such that $\delta$ is 
Woodin in 
$L(V_\delta,W\rest V_\delta,S,T)$ (so $\delta_\eta\leq\delta_2$).
 Say that an iteration tree $\Tt$ is \emph{$W$-nice}
 if $E^\Tt_\alpha$ is nice and $E^\Tt_\alpha$
 coheres $W$ through $\varrho(E^\Tt_\alpha)$
 for each $\alpha+1<\lh(\Tt)$.
\begin{clm}\label{clm:unique_cof_wfd_branch}
Let $\eta<\delta_2$ and $\eta<\xi<\delta_\eta$. Let $G$ be $(V,\Coll(\om,\eta))$-generic.
Let $\Tt\in V[G]$ be a limit length $W$-nice iteration tree on $V$ which is based on the interval $(\eta,\xi)$ and has length $\leq\eta$. Then:
\begin{enumerate}
\item\label{item:V[G]_sats_exactly_one_cof_wfd_branch}   $V[G]\sats$``$\Tt$ has  exactly one cofinal wellfounded branch''.
\item\label{item:V_sats_exactly_one_cof_wfd_branch} If $\Tt\in V$ then $V\sats$``$\Tt$ has exactly one cofinal wellfounded branch''.
\end{enumerate}

\end{clm}
\begin{proof}
Part \ref{item:V[G]_sats_exactly_one_cof_wfd_branch}:
\begin{sclm}
 $V[G]\sats$``$\Tt$ has at most one cofinal wellfounded branch''.
\end{sclm}
\begin{proof}
Suppose $b,c$ are distinct $\Tt$-cofinal wellfounded branches. Let $M_b=M^\Tt_b$ and $M_c=M^\Tt_c$.
Because $\Tt$ is based on $(\eta,\xi)$ we have $\delta(\Tt)\leq 
i^\Tt_b(\xi)<i^\Tt_b(\delta_\eta)$. Our choice of $\delta_\eta$ therefore gives that
\begin{equation}\label{eqn:not_Woodin} M_b\sats\text{``}L(V_{\delta(\Tt)},W^{M_b}\rest V_{\delta(\Tt)},S^{M_b},T^{M_b})\sats
\text{``}\delta(\Tt)\text{ is not Woodin''.''} \end{equation}
Likewise for $M_c,i^\Tt_c$. But by Claim \ref{clm:S,T_fixed},
\[ (S^{M_b},T^{M_b})=(S,T)=(S^{M_c},T^{M_c}).\]
We have $V'=V_{\delta(\Tt)}^{M_b}=V_{\delta(\Tt)}^{M_c}$,
and because $\Tt$ is $W$-nice,
$W'=W^{M_b}\rest V'=W^{M_c}\rest V'$.
So
\[ L(V',W',S,T)\sub M_b\inter M_c, \]
so by the Zipper Lemma,
$L(V',W',S,T)\sats\text{``}\delta(\Tt)\text{ is Woodin''}$,
contradicting line (\ref{eqn:not_Woodin}).\end{proof}

\begin{sclm}\label{sclm:cofinal_wfd_branch}
 $V[G]\sats$``$\Tt$ has a cofinal wellfounded branch''.
\end{sclm}
\begin{proof} This is an almost standard fact.
It is almost proved in \cite{itertrees},
but not quite; here is the rest of the proof:

It suffices to see that the tree $\Tt^+$ on $V[G]$, with $\Tt^+$ equivalent to $\Tt$, as in the proof of Claim \ref{clm:S,T_fixed}, has a cofinal wellfounded branch in $V[G]$. So suppose otherwise. Work in $V[G]$.
Then $\Tt^+$ is continuously illfounded, 
by \cite[Theorem 5.6]{itertrees},
and since $\Tt$ and hence also $\Tt^+$ are extender-nice trees, hence $2^{\aleph_0}$-closed (in $V$ and $V[G]$ respectively), and since $\eta$ is countable. Moreover, by a small variant of the proof of the same result, for each limit $\lambda<\lh(\Tt^+)$, $\Tt^+\rest\lambda$
is continuously illfounded off $[0,\lambda)^{\Tt^+}$. Fix a sequence $\Pi=\left<\pi_\lambda\right>_{\lambda\in\mathrm{Lim}\cap(\lh(\Tt^+)+1)}$
of functions witnessing the continuous illfoundedness (off $[0,\lambda)^{\Tt^+}$ in case $\lambda<\lh(\Tt^+)$). Let $\sigma:M\to V_\alpha[G]$
be elementary where $\alpha$ is large enough and $(V_\alpha[G],\delta_2)$ is a coarse premouse,
$M$ is countable and transitive,
and $\Tt^+,\Pi\in\rg(\sigma)$. Let $\sigma(\Uu^+)=\Tt^+$. Then by \cite[Theorem 4.3]{itertrees}, or more generally, by \cite{plus_0_trees}
in case $\Tt^+$ is not a plus 2 tree, there
is a $\Uu^+$-maximal branch $b$ which is $\sigma$-realizable. But the continuous illfoundedness, reflected by $\sigma$, clearly gives that $M^{\Uu^+}_b$ is illfounded, a contradiction.
\end{proof}

Part \ref{item:V_sats_exactly_one_cof_wfd_branch}:
This is an immediate consequence of part \ref{item:V[G]_sats_exactly_one_cof_wfd_branch}
and the homogeneity of the collapse.
\end{proof}

Say a partial $\alpha$-iteration strategy is \emph{$W$-nice} if its domain includes all $W$-nice trees
 of length $<\alpha$.

\begin{clm}\label{clm:V_iterable}
Let $\eta<\delta_2$ and $\eta<\xi<\delta_\eta$.
Then $V$ is $W$-nicely $(\eta+1)$-iterable for 
trees 
based on the interval 
$(\eta,\xi)$,
via the strategy  which chooses unique cofinal wellfounded branches.
\end{clm}
\begin{proof}
 By Claim \ref{clm:unique_cof_wfd_branch},
 this does not break down at limit stages.
 But if there is a tree $\Tt$ of length $\lambda+1\leq\eta$ which extends to a putative tree $\Tt'$ of length
$\lambda+2$ with $M^{\Tt'}_{\lambda+1}$ illfounded, then we get a contradiction much as in the proof of Subclaim \ref{sclm:cofinal_wfd_branch} of the proof of Claim \ref{clm:unique_cof_wfd_branch}.
\end{proof}

Now let $\omega_1\leq\eta<\delta_2$. Working in $W=L(V_{\delta_\eta},W\rest V_{\delta_\eta},S,T)$,
let $\theta\gg\gamma$ (recall $S,T$ are on $\om\cross\gamma$) be such that $L_\theta(V_{\delta_\eta},W\rest V_{\delta_\eta},S,T)$
is satisfies the theory specified in the first bullet point of \cite[Definition 3.1]{iter_for_stacks}.\footnote{For the reader's convenience, that is, writing $N=L_\theta(V_{\delta_\eta},S,T)$,
then $\delta_\eta\leq\OR^N=\rank(N)$, $\delta_\eta$ and $\OR^N$ are limit ordinals, $\card^N(V_\alpha^N)<\delta_\eta$ for all $\alpha<\delta_\eta$, $\cof^N(\delta_\eta)$ is not measurable in $N$, $N$ satisfies $\Sigma_0$-comprehension and is rudimentarily closed,
and $N$ satisfies $\lambda$-choice for all $\lambda<\delta_\eta$.}
Let $\pi_\eta:M_\eta\to 
L_\theta(V_{\delta_\eta},W\rest V_{\delta_\eta},S,T)$ be elementary
with $\eta,\delta_\eta,W\rest V_{\delta_\eta},S,T\in\rg(\pi_\eta)$ and $M_\eta$ 
countable. Let $\pi_\eta(\eta^M,\delta^M,W^M,S^M,T^M)=(\eta,\delta_\eta,W\rest V_{\delta_\eta},S,T)$. Let $\xi=\sup\pi_\eta``\delta^M$. Then $\xi<\delta_\eta$ because $\pi_\eta\in W$. 
Let $\Sigma_\eta$ be the above-$(\eta^M+1)$, $W^M$-nice $(\eta+1)$-iteration strategy for $M$
given by lifting to wellfounded trees $\pi\Tt$ on $V$. Note here that all such $\pi\Tt$ are  $W$-nice and based on the 
interval $(\eta,\xi)$, so Claim \ref{clm:V_iterable} applies. The entire sequence $\left<M_\eta,\pi_\eta\right>_{\omega_1\leq\eta<\delta_2}$ is chosen in $V$, where AC holds.

\begin{clm}
The restriction of $\Sigma_\eta$ to countable length trees is $\eta$-homogeneously Suslin.
\end{clm}
\begin{proof}
Use the generalization for the Wind{\ss}us theorem to arbitrary length trees.
Use finite substructures of iteration trees instead of finite initial segments.
It is similar to Martin-Steel \cite[Theorem 5.6]{itertrees},
in the case of length $>\om$. Here are some more details. We need to consider codes for iteration trees, consisting of a pair 
$(w,t)$, where $w\in\WO$ and $t$ codes the tree, of length $|w|$.
Since there is a measurable cardinal $>\eta$, $\WO$ is $\eta$-hom Suslin.
Then fixing $n<\om$, consider the set $W_n$ of pairs $(w,t)$ such that $w\in\WO$ 
and $t$ is a code for a ``pseudo iteration trees'' $\Tt$ on $M$ of length $|w|$
such that $M^{\pi\Tt}_{|w,n|}$ is wellfounded (where $\pi\Tt$ is also a ``pseudo iteration tree''). 
Here letting $\sigma:\om\to|w|$ be the natural bijection,
$|w,n|=\sigma(n)$. And a ``pseudo iteration tree'' may have 
many illfounded models, but otherwise it is like an iteration tree. We consider the model indexed 
at 
$n$ as a direct limit of models indexed on finite trees,
and use the method of Wind{\ss}us' proof to get that the set of such codes for which
$M^{\pi\Tt}_{|w,n|}$ is wellfounded, is $\eta$-hom Suslin (here we also use the fact that the 
$\eta$-hom Suslin sets are closed under intersection, to require that $w\in\WO$). Finally, 
$\eta$-hom Suslin is closed under countable 
intersection, and the desired set $W$ is just $\bigcap_{n<\om}W_n$, so we are done.
\end{proof}

Let $X\sub[\omega_1,\delta_2)$ be cofinal in $\delta_2$ and such that for all $\eta_0,\eta_1\in X$,
we have $M_{\eta_0}=M_{\eta_1}$ and 
$(\eta,\delta,W,S,T)^{M_{\eta_0}}=(\eta,\delta,W,S,T)^{M_{\eta_1}}$ 
and $\Sigma_{\eta_0}\rest\HC=\Sigma_{\eta_1}\rest\HC$.
Let $\Sigma=\bigcup_{\eta\in X}\Sigma_\eta$ and $M=M_\eta$ for $\eta\in X$.
\begin{clm}
$\Sigma$ is a $W^M$-nice $\delta_2$-iteration strategy for $M$, $\Sigma$ has strong hull
condensation, and $\Sigma\rest\HC$ is $<\delta_2$-homogeneously Suslin.
\end{clm}
\begin{proof}
First note that each $\Sigma_\eta$ has strong hull condensation,
because $\Sigma_\eta$ lifts to unique wellfounded branches on $V$.

So it suffices to see that for all $\eta_0,\eta_1\in X$, if $\eta_0<\eta_1$ then 
$\Sigma_{\eta_0}\sub\Sigma_{\eta_1}$. Suppose not and let $\Tt$ on $M$ be according to both $\Sigma_{\eta_0}$ and $\Sigma_{\eta_1}$,
of limit length $\leq\eta_0$,
and such that $\Sigma_{\eta_0}(\Tt)=b\neq c=\Sigma_{\eta_1}(\Tt)$.
 Let $\pi:H\to V_\theta$ be elementary with
 $H$ countable and transitive and everything relevant in $\rg(\pi)$.
 Let $\pi(\bar{\Tt},\bar{b},\bar{c})=(\Tt,b,c)$. So $\bar{\Tt}$ is also on $M$,
 but has countable length, and $\bar{b},\bar{c}$ are $\bar{\Tt}$-cofinal with $\bar{b}\neq\bar{c}$.
 
 \begin{sclm}
 $\bar{\Tt}\conc\bar{b}$ is via $\Sigma_{\eta_0}$ and $\bar{\Tt}\conc\bar{c}$ is via $\Sigma_{\eta_1}$.\end{sclm}
 \begin{proof}Consider $\bar{\Tt}\conc\bar{b}$. The point is that the $\pi_{\eta_0}$-copy $\bar{\Uu}\conc\bar{b}$ of $\bar{\Tt}\conc\bar{b}$ has wellfounded models,
because the $\pi_{\eta_0}$-copy $\Uu\conc b$ of $\Tt\conc b$
has wellfounded models,
and $\bar{\Uu}\conc\bar{b}$ is a hull of $\Uu\conc b$ in a natural manner.

Here are some details regarding
the the last statement. Let
\[ \varphi:(\lh(\bar{\Tt})+1)\to(\lh(\Tt)+1) \]
be  $\varphi=\pi\rest(\lh(\bar{\Tt})+1)$.
Define elementary embeddings
\[ \sigma_\alpha:M^{\bar{\Uu}\conc\bar{b}}_\alpha\to M^{\Uu\conc b}_{\varphi(\alpha)} \]
for $\alpha\leq\lh(\bar{\Tt})$,
by setting $\sigma_0=\id:V\to V$,
noting that $E^{\Uu\conc b}_{\varphi(\alpha)}=\sigma_\alpha(E^{\bar{\Uu}\conc\bar{b}}_\alpha)$ for $\alpha+1\leq\lh(\bar{\Tt})$,
defining $\sigma_{\alpha+1}$ via the Shift Lemma,
and for limit $\alpha$,
defining $\sigma_\alpha$ as the unique map which commutes with the iteration maps and all maps $\sigma_\beta$ for $\beta<^{\bar{\Uu}\conc\bar{b}}\alpha$. We have the same situation regarding $\bar{\Tt}\conc\bar{b}$ and $\Tt\conc b$,
with embeddings
\[ \varrho_\alpha:M^{\bar{\Tt}\conc\bar{b}}_\alpha\to M^{\Tt\conc b}_{\varphi(\alpha)},\]
and moreover,
\[ \varrho_\alpha=\pi\rest M^{\bar{\Tt}\conc\bar{b}}_\alpha.\]
The hull embeddings commute with the iteration embeddings, i.e.
\[ \varrho_\alpha\com i^{\bar{\Tt}\conc\bar{b}}_{\beta\alpha}=i^{\Tt\conc b}_{\varphi(\beta)\varphi(\alpha)}\com\varrho_\beta \]
and
\[ \sigma_\alpha\com i^{\bar{\Uu}\conc\bar{b}}_{\beta\alpha}=i^{\Uu\conc b}_{\varphi(\beta)\varphi(\alpha)}\com\sigma_\beta\]
whenever $\beta\leq^{\bar{\Tt}\conc\bar{b}}\alpha$. And of course the copy maps also commute with the iteration embeddings, i.e.
letting \[\Pi_\alpha:M^{\bar{\Tt}\conc\bar{b}}_\alpha\to M^{\bar{\Uu}\conc\bar{b}}_\alpha \]
(for $\alpha\leq\lh(\bar{\Tt})$) and 
\[\Psi_\gamma:M^{\Tt\conc b}_\gamma\to M^{\Uu\conc b}_\gamma \] (for $\gamma\leq\lh(\Tt)$)  be the copy maps, then
\[ \Pi_\alpha\com i^{\bar{\Tt}\conc\bar{b}}_{\beta\alpha}=i^{\Tt\conc b}_{\varphi(\beta)\varphi(\alpha)}\com\Pi_\beta \]
and
\[ \Psi_\alpha\com i^{\bar{\Uu}\conc\bar{b}}_{\beta\alpha}=i^{\Uu\conc b}_{\varphi(\beta)\varphi(\alpha)}\com\Psi_\beta\]
whenever $\beta\leq^{\bar{\Tt}\conc\bar{b}}\alpha$. It is straightforward to maintain this hypotheses, and to see that the maps $\sigma_\alpha,\sigma_\beta$ agree with one another appropriately that the Shift Lemma applies. We leave the remaining details to the reader.

Since  $\Uu\conc b$ has wellfounded models,
and we have the maps $\sigma_\alpha$,
$\bar{\Uu}\conc\bar{b}$ also has wellfounded models,
and therefore $\bar{\Tt}\conc\bar{b}$ is via $\Sigma_{\eta_0}$, as desired.

The fact that $\bar{\Tt}\conc\bar{c}$ is via $\Sigma_{\eta_1}$ is likewise.
\end{proof}

But  $\bar{\Tt}$ has countable length and
$\Sigma_{\eta_0}\rest\HC=\Sigma_{\eta_1}\rest\HC$. So $\bar{b}=\bar{c}$, contradiction.
\end{proof}

We have established parts \ref{item:Sigma_is_strategy}--\ref{item:M_has_trees} of 
\ref{lem:capturing_mouse},
and the claim below gives part \ref{item:capturing_trees}:
\begin{clm}
Let $\Tt,P,G$ be as in part \ref{item:capturing_trees}. Then
$p[S^P]^{V[G]}\sub p[U]^{V[G]}$ and likewise for $T^P,W$.
\end{clm}
\begin{proof}
Let $\kappa+1=\lh(\Tt)$ and $\eta\in X\cut\kappa$
and $\Uu=\pi_\eta\Tt$, a tree on $V$ which is above $\eta+1$. Work in $V[G]$,
where $\kappa$ is countable. Then $[0,\kappa]^\Tt$ is $\pi_\eta$-realizable;
i.e. there is 
 an elementary $\sigma:M^\Tt_\infty\to M^\Uu_\infty$
such that $\sigma\com i^\Tt_{0\infty}=\pi_\eta$.
(For let $\Uu^+$ be the tree on $V[G]$ equivalent to $\Uu$. We have the copy map \[\sigma':M^\Tt_\infty\to i^{\Uu^+}_{0\infty}(L_\theta(V_{\delta_\eta},S,T)), \]
and $\sigma'\com i^{\Tt}_{0\infty}=i^{\Uu^+}_{0\infty}(\pi_\eta)$. By absoluteness
and since $\Tt$ is countable in $M^{\Uu^+}_\infty$,
there is a map $\sigma''\in M^{\Uu^+}_\infty$
with the same properties. This pulls back to $V[G]$ under the elementarity of $i^{\Uu^+}_{0\infty}$.)

So $\sigma(S^P,T^P)=(S,T)$, but then
\[ V[G]\sats\text{``}p[S^P]\sub p[S]=p[U]\text{''},\]
and likewise for $T^P,T,W$, as required.
\end{proof}

This completes the proof.
\end{proof}

We now proceed to the scale construction:
\begin{proof}[Proof of Corollary \ref{cor:local} (stationary tower free)]
Fix $M,\Sigma$ witnessing \ref{lem:capturing_mouse} (with respect to $\delta_2,\delta_3,U,W$).
It suffices to see that there is a scale on $p[U]$ which is $\Delta^1_2(\Sigma\rest\HC)$.

Write ${<^M}=W^M$.
Let $\mathscr{E}^M$ be the set of $E\in V_{\delta^M}^M$
such that $M\sats$``$E$ is a  ${<^M}$-nice extender''. 
Given  $E,F\in\mathscr{E}^M$, say that $E<^M_eF$ iff
either $M\sats$``the Mitchell order rank of $E$ is strictly less than the Mitchell order rank of 
$F$'', or $M\sats$``$E,F$ have the same Mitchell order rank, and $E<^MF$''.
So we have $(*)_M$, which asserts:
\begin{enumerate}[label=--]
 \item ${<^M_e}\in M$, 
 \item $<^M_e$ well orders the ${<^M}$-nice extenders of $M$, and
\item if $F\in M$ is ${<^M}$-nice and $E\in\Ult(M,F)\sats$``$E$ is ${<^{\Ult(M,F)}}$-nice'' and 
\[ \strength^{\Ult(M,F)}(E)\leq\strength^M(F) \]
then $E\in M\sats$``$E$ is ${<^M}$-nice'' and $E<^M_eF$.
\end{enumerate}
Given an iterate $P$ of $M$, let ${<^P_e}=i_{MP}({<^M_e})$. Clearly then $(*)_P$ holds.

We have verified that $(M,\delta^M,\mathscr{E}^M,{<^M_e})$ is a slightly coherent weak coarse premouse (as in \cite[Definition 3.1]{iter_for_stacks},
and see \cite[p.~8, \S1.1.2]{iter_for_stacks}
for the definition of \emph{suitable} extender).

\begin{rem} We will deal with inflations, particularly terminal inflations; 
see \cite{iter_for_stacks}. 
 Let $\Tt,\Xx$ 
be successor length trees on $M$, via $\Sigma_M$, such that $\Xx$ is a terminal inflation of $\Tt$. 
That is, the last node of $\Xx$ is associated to the last node of $\Tt$,
and in particular, we have a canonical final copy map
\[ \sigma^{\Tt,\Xx}:M^\Tt_\infty\to M^\Xx_\infty. \]
If $\Yy$, of successor length on $M$, via $\Sigma_M$, is a terminal inflation of $\Xx$, then 
$\sigma^{\Xx,\Yy}\com\sigma^{\Tt,\Xx}=\sigma^{\Tt,\Yy}$.
If $\left<\Xx_n\right>_{n<\om}$, each of successor length on $M$, via $\Sigma_M$, are such that 
$\Xx_{n+1}$ is a terminal inflation of $\Xx_n$, then we have the inflationary comparison
$\Xx_\om$ of the sequence $\left<\Xx_n\right>_{n<\om}$ (see \cite{iter_for_stacks}). Here $\Xx_\om$
is a terminal inflation of each $\Xx_n$, and
$\sigma^{\Xx_n,\Xx_\om}\com\sigma^{\Xx_m,\Xx_n}=\sigma^{\Xx_m,\Xx_\om}$.)
\end{rem}

The \emph{$<^M$-nice extender algebra}
is the version of the extender algebra of $M$
at $\delta^M$, in which we only use ${<^M}$-nice
extenders to induce axioms. Likewise for iterates $P$ of $M$.

\begin{clm}\label{lem:gen_inf} Let $\Tt$ be any successor length ${<^M}$-nice tree on $M$, via $\Sigma_M$. 
Let $x\in\RR$.
Then there is a ${<^M}$-nice $\Xx$ via $\Sigma_M$,
such that $\Xx$ is a terminal inflation of $\Tt$,  via $\Sigma_M$, such that 
$x$ is ${<^{M^\Xx_\infty}}$-nice extender algebra generic over $M^\Xx_\infty$ at $\delta^{M^\Xx_\infty}$.\end{clm}

\begin{proof} By \cite[Theorem 5.6] {iter_for_stacks}, noting that we only ever need to use ${<^{M^\Xx_\gamma}}$-nice extenders as inflationary extenders, by definition of the ${<^{M^\Xx_\gamma}}$-nice extender algebra.
\end{proof}

\begin{clm}\label{clm:stabilize_lx}
Let $x\in p[T]$.
For every ${<^M}$-nice countable successor length nice tree $\Tt$ via $\Sigma$ there is a
${<^M}$-nice countable successor length  
tree $\Uu$ via $\Sigma$ such that $(\Tt,\Uu)$ is a  terminal inflation, $x\in 
p[T^{M^\Uu_\infty}]$, and whenever $(\Uu,\Xx)$ is a terminal inflation with $\Xx$ being ${<^M}$-nice and via $\Sigma$,
then letting $\sigma=\sigma^{\Uu,\Xx}$ and $l^\Uu_x=\leftmost(T^{M^\Uu_\infty}_x)$ (the left-most branch through the $x$-section of $T^{M^{\Uu_\infty}}$), we have
\[ \sigma``l^\Uu_x\sub T^{M^\Xx_\infty}_x\text{ and 
}\sigma``l^\Uu_x=\leftmost(T^{M^\Xx_\infty}_x).\]
\end{clm}
\begin{proof}
Suppose not. Then we can define a sequence $\left<\Tt_\alpha\right>_{\alpha\leq\om_1}$ such that:
\begin{enumerate}
 \item $\Tt_\alpha$ is a ${<^M}$-nice, successor length tree via $\Sigma_M$; write 
$T^\alpha=T^{M^{\Tt_\alpha}_\infty}$ and $l^\alpha_x=\leftmost(T^\alpha_x)$ given $x\in 
p[T^\alpha]$,
\item if $\alpha<\om_1$ then $\Tt_\alpha$ has countable length,
 \item $x\in p[T^0]$,
 \item $(\Tt_{\alpha+1},\Tt_\alpha)$ is a terminal inflation for all $\alpha<\om_1$,
 \item for limit $\lambda$, $\Tt_\lambda$ is comparison inflation of 
$\left\{\Tt_\alpha\right\}_{\alpha<\lambda}$,
\item thus, for $\alpha\leq\om_1$, we have $x\in p[T^\alpha]$, and for $\alpha<\beta\leq\om_1$,
$(\Tt_\alpha,\Tt_\beta)$ is a  terminal inflation; write 
$\sigma^{\alpha\beta}=\sigma^{\Tt_\alpha,\Tt_\beta}$,
\item $l^{\alpha+1}_x<_\lex\sigma^{\alpha,\alpha+1}``l^\alpha_x$, for all $\alpha<\om_1$,
\item thus, $l^\beta_x\leq_\lex\sigma^{\alpha\beta}`` l^\alpha_x$ for
all $\alpha<\beta\leq\om_1$.
\end{enumerate}

We start by getting $\Tt_0$ with $x\in p[T_0]$ by using Claim \ref{lem:gen_inf}.
The remainder of the sequence is produced by using the contradictory hypothesis
and the existence of  simultaneous inflations at limit stages.

Now $M^{\Xx_{\om_1}}_\infty$ is wellfounded and for each $\alpha<\beta<\om_1$,
\[ \sigma^{\alpha\om_1}=\sigma^{\beta\om_1}\com\sigma^{\alpha\beta}.\]
But then an easy induction on $n<\om$ shows that there is $\alpha_n<\om_1$ such that for all 
$\alpha\in(\alpha_n,\om_1)$,
\[ \sigma^{\alpha_n\alpha}(l^{\alpha_n}_x(n))=l^{\alpha}_x(n). \]
But since $\sup_{n<\om}\alpha_n<\om_1$, this clearly gives a contradiction.
\end{proof}
\begin{clm}\label{clm:compare_lx,ly}
 For any ${<^M}$-nice, countable successor length trees $\Tt,\Uu$ via $\Sigma_M$ there is a ${<^M}$-nice 
countable $\Xx$ via $\Sigma_M$ such that both $(\Xx,\Tt)$ and $(\Xx,\Uu)$ are terminal  
inflations.
\end{clm}
\begin{proof}
Define $\Xx$ to be the comparison inflation of $\{\Tt,\Uu\}$; see \cite[\S5.1]{iter_for_stacks}.
\end{proof}

We now define the scale on $p[T]$. Given $x,y\in p[T]$,
we set $x\leq_n y$ iff for every countable ${<^M}$-nice successor length $\Tt$ via $\Sigma$
there is a countable ${<^M}$-nice $\Xx$ via $\Sigma$ and such that $(\Xx,\Tt)$ is a terminal inflation
and $(x,y)$ is extender algebra generic over $M^\Xx_\infty$ at $\delta^{M^\Xx_\infty}$, and letting 
$l_x=\leftmost(T^{M^\Xx_\infty}_x)$ and $l_y$ likewise, then
\[ 
(x(0),l_x(0),\ldots,x(n-1),l_x(n-1))\leq_\lex(y(0),l_y(0),\ldots,y(n-1),l_y(n-1)).\]

\begin{clm}
For each $n<\om$, ${\leq_n}\rest p[T]$ is a prewellorder of $p[T]$.
\end{clm}
\begin{proof}
Fix $n<\om$. Clearly $\leq_n\rest p[T]$ is reflexive,
and using Claim \ref{clm:stabilize_lx}, it is easy to see that it is transitive (and cf.~the proof to follow). For comparitivity,
let $x,y\in p[T]$ and let $\Tt$ be a countable successor length ${<^M}$-nice tree  via $\Sigma$. Using Claims \ref{clm:stabilize_lx} and \ref{clm:compare_lx,ly}, we can find a countable ${<^M}$-nice $\Xx$ such that $(\Tt,\Xx)$ simultaneously witnesses Claim \ref{clm:stabilize_lx}
for both $x$ and $y$, and such that $(x,y)$ is extender algebra generic over $M^\Xx_\infty$ at $\delta^{M^\Xx_\infty}$. Let $(\Tt',\Xx')$ be likewise. But then letting $l_x=\mathrm{left}(T^{M^\Xx_\infty}_x)$ and $l_y$ likewise,
and $l'_x=\mathrm{left}(T^{M^{\Xx'}_\infty})$ and $l'_y$ likewise, we have
\[ (x(0),l_x(0),\ldots,x(n-1),l_x(n-1))\leq_{\mathrm{lex}}(y(0),l_y(0),\ldots,y(n-1),l_y(n-1)) \]
iff
\[ (x(0),l'_x(0),\ldots,x(n-1),l'_x(n-1))\leq_{\mathrm{lex}}(y(0),l'_y(0),\ldots,y(n-1),l'_y(n-1)), \]
since by Claim \ref{clm:compare_lx,ly}, we can find $\Xx''$
such that $(\Xx,\Xx'')$ and $(\Xx',\Xx'')$
are terminal  inflations, and by the properties given by Claim \ref{clm:stabilize_lx}.

Finally, ${\leq_n}\rest p[T]$ is wellfounded, by similar kinds of considerations. (Given $\left<x_n\right>_{n<\om}$
and given a countable $\Tt$, note that we can find a countable $\Xx$ such that $(\Tt,\Xx)$
is a terminal  inflation and $\Xx$ simultaneously witnesses Claim \ref{clm:stabilize_lx} for all $x_n$.)\end{proof}

Let $\varphi_n(x)$ be the $\leq_n$-rank of $x\in p[T]$.

\begin{clm} $\left<{\leq_n}\right>_{n<\om}$ is a scale on $p[T]$.
 \end{clm}
\begin{proof} 
Let $\left<x_\ell\right>_{\ell<\om}$ be such that for all $n<\om$, $\varphi_n(x_\ell)$ is constant for large $\ell<\om$. Let $x=\lim_{\ell\to\om}x_\ell$. Let $\Tt$ be any countable ${<^M}$-nice tree via $\Sigma$, and let $\Xx$ be such that $(\Tt,\Xx)$
is a terminal inflation and $\Xx$ witnesses Claim \ref{clm:stabilize_lx} simultaneously for all $x_n$. (We get this by producing a sequence $\left<\Xx_n\right>_{n<\om}$ of trees such that $(\Xx_n,\Xx_{n+1})$ is a terminal inflation witnessing the claim with respect to $x_n$,
and then letting $\Xx$ be the comparison inflation of $\{\Xx_n\}_{n<\om}$.) Let $l_{x_n}=\mathrm{left}(T^{M^\Xx_\infty}_{x_n})$. Then note that the limit $l$ of the branches $l_{x_n}$ exists, and $(x,l)\in [T^{M^\Xx_\infty}]$. So by the properties of Lemma \ref{lem:capturing_mouse},
we have $x\in p[T]$. So $\left<{\leq}_n\right>_{n<\om}$ is a semiscale.
Finally, by choice of $\Xx$ and the definition of $l$, we easily get lower semicontinuity,
so it is a scale.
\end{proof}

Clearly $\left<{\leq}_n\right>_{n<\om}$ is $\Delta^1_2(\Sigma\rest\HC)$, so we are done.
\end{proof}

\bibliographystyle{plain}
\bibliography{../bibliography/bibliography}

\end{document}